\documentclass[a4paper,11pt]{article}

\usepackage[T1]{fontenc}
\usepackage{lmodern}
\usepackage{amsmath,amssymb,amsthm,mathtools}

\usepackage{enumitem,needspace}
\usepackage[colorlinks=true,linkcolor=black,citecolor=black,urlcolor=black]{hyperref}
\hypersetup{
  pdftitle={Core Coordinates and Negative Support in A-Hypergeometric Systems},
  pdfauthor={Mao Nagamine}
}

\newcommand{\CC}{\mathbb C}
\newcommand{\QQ}{\mathbb Q}
\newcommand{\RR}{\mathbb R}
\newcommand{\ZZ}{\mathbb Z}
\newcommand{\NN}{\mathbb N}
\newcommand{\nsupp}{\operatorname{nsupp}}
\newcommand{\supp}{\operatorname{supp}}
\newcommand{\cone}{\operatorname{Cone}}
\newcommand{\inte}{\operatorname{int}}
\newcommand{\inw}{\operatorname{in}_{w}}
\newcommand{\NS}{\operatorname{NS}_{w}}

\newtheorem{theorem}{Theorem}[section]
\newtheorem{proposition}[theorem]{Proposition}
\newtheorem{lemma}[theorem]{Lemma}
\newtheorem{corollary}[theorem]{Corollary}
\theoremstyle{definition}
\newtheorem{definition}[theorem]{Definition}

\newtheorem{remark}[theorem]{Remark}
\numberwithin{equation}{section}

\title{Core Coordinates and Negative Support\\
in $A$-Hypergeometric Systems}
\author{Mao Nagamine
\thanks{This work was supported by JST SPRING, Grant Number JPMJSP2119.}}
\date{}

\begin{document}
\maketitle

\begin{abstract}
For a homogeneous $A$-hypergeometric system, let $K_v$ be the
intersection of the negative support family for a
fake exponent $v$, and let $C_w$ be the set of indices of variables that do not occur in any minimal monomial generator of $\operatorname{in}_w(I_A)$.
We prove that $K_v=\operatorname{nsupp}(v)\cap C_w$ holds for every
parameter and every fake exponent if and only if
$\operatorname{Cone}(A_{C_w})$ meets the interior of
$\operatorname{Cone}(A)$.
We give equivalent criteria using the reduced Gr\"obner basis and the
associated regular triangulation.
When the condition fails, we construct an infinite family of fake
exponents for which the equality fails.
\end{abstract}

\begingroup\small
\par\endgroup

\section{Introduction}

The negative support of a fake exponent plays an important role in the
construction of $A$-hypergeometric series, especially those involving
logarithms.
Saito \cite{Saito2002} studied logarithm-free $A$-hypergeometric series
and gave a dimension formula for their solution space.
In \cite{Saito2020}, he developed a version of Frobenius's method that
perturbs a fake exponent in directions belonging to
$L=\ker_{\ZZ}A$.
Okuyama and Saito \cite{OS2} further developed this method by comparing
spaces of differential operators with constant coefficients and
obtained conditions for constructing a fundamental system of series
solutions with a fixed exponent.
More recently, they described spaces of logarithmic coefficients by
colon ideals and compared perturbations along vectors in $\ker_{\CC}A$ with
perturbations in the full ambient space \cite{OS3}.
The present paper concerns the intersection of the negative supports
used in these constructions.

Fix a generic weight $w$ and write the reduced Gr\"obner basis of the
toric ideal $I_A$ as
\[
\partial^{g^{(k)}_+}-\partial^{g^{(k)}_-},
\]
with the first monomial as the initial term.
These difference vectors generate a subsemigroup $C(w)$ of $L$.
For a fake exponent $v$, Saito \cite[Equation~(9)]{Saito2020}
introduced the family of negative supports for which every lattice
shift realizing the same negative support belongs to $C(w)$.
We denote this family by $\NS(v)$ and set
\[
K_v=\bigcap_{I\in\NS(v)} I.
\]
Let $C_w$ be the set of indices of variables that occur in none of the
minimal monomial generators of
\[
M=\operatorname{in}_w(I_A).
\]
Equivalently, $C_w$ is the intersection of the free index sets of all
standard pairs of $M$, including embedded standard pairs.
\cite[Theorem~3.1]{NagamineCore} proved that
\[
\operatorname{nsupp}(v)\subseteq C_w
\quad\Longrightarrow\quad
K_v=\operatorname{nsupp}(v).
\]
The same argument gives, for an arbitrary fake exponent,
\[
\operatorname{nsupp}(v)\cap C_w
\subseteq K_v
\subseteq\operatorname{nsupp}(v).
\]
This raises the question of when the first inclusion is an equality.

Our main result gives a necessary and sufficient condition for this
equality to hold for every parameter and every fake exponent.
Writing $a_1,\ldots,a_n$ for the columns of $A$, we prove that
\[
K_v=\operatorname{nsupp}(v)\cap C_w
\]
holds for all such $v$ if and only if
\[
\operatorname{Cone}(A_{C_w})
\cap\operatorname{int}\operatorname{Cone}(A)
\ne\emptyset.
\]
This condition is also equivalent to the absence of a proper face of
$\operatorname{Cone}(A)$ containing all core columns, and to the
existence of an $h\in C(w)$ such that
\[
h_i>0\quad(i\notin C_w).
\]
Here a proper face means a face different from the whole cone.

For sufficiency, the existence of such a vector $h$ allows us to eliminate the negative support outside the core by lattice shifts.
For necessity, suppose that all core columns lie in a proper face of
$\operatorname{Cone}(A)$.
We use a boundary simplex of codimension one in the associated regular
triangulation to construct fake exponents satisfying
\[
K_v=\operatorname{nsupp}(v)=\{j\},
\quad j\notin C_w.
\]
The construction gives an infinite family of counter-examples whenever
the geometric condition fails.

The condition can also be checked directly from the initial ideal and
its regular triangulation $\Delta$.
If a boundary simplex $\rho$ of dimension $d-2$ belongs to a unique
maximal face
\[
\sigma=\rho\cup\{j\},
\]
then the condition is equivalent to requiring $j\in C_w$ for every
such $\rho$.
Thus the criterion can be verified by a finite computation from
$\operatorname{in}_w(I_A)$ and $\Delta$.

Section~2 recalls negative support, standard pairs, the core index set
$C_w$, and regular triangulations.
Section~3 proves the main theorem and gives a criterion in terms of
the initial ideal.
Section~4 presents a square configuration satisfying the condition and
a hexagonal configuration for which it fails.

\section{Negative support and standard monomials}\label{sec:prelim}

\subsection{Notation and the negative support family}

Let
$$
A=(a_1,\ldots,a_n)\in\ZZ^{d\times n}
$$
be of rank $d$.  We assume throughout that the columns of $A$ lie in an
affine hyperplane of $\mathbb{Q}^d$ not containing the origin.
Set $\NN=\{0,1,2,\ldots\}$ and
\[
L=\ker_{\ZZ}A,\quad
I_A=\langle\partial^a-\partial^b\mid
a,b\in\NN^n,\ Aa=Ab\rangle
\subseteq\CC[\partial_1,\ldots,\partial_n].
\]
In the Weyl algebra
$D=\CC\langle x_1,\ldots,x_n,\partial_1,\ldots,\partial_n\rangle$,
the $A$-hypergeometric ideal with parameter $\beta\in\CC^d$ is
\[
H_A(\beta)=DI_A+\sum_{i=1}^d D(E_i-\beta_i),
\quad E_i=\sum_{j=1}^n a_{ij}x_j\partial_j.
\]
A formal series annihilated by all operators in this ideal is called an
$A$-hypergeometric series.
For $u\in\ZZ^n$, write $u=u_+-u_-$, where $u_+,u_-$ are nonnegative
vectors with disjoint supports.
Fix a weight $w\in\RR^n$ for which $M=\inw I_A$ is a monomial ideal,
and write the reduced Gr\"obner basis as
\begin{equation}\label{eq:GB}
f_k=\partial^{g^{(k)}_+}-\partial^{g^{(k)}_-},
\quad 1\le k\le m.
\end{equation}
The first monomial is the initial term, hence $w\cdot g^{(k)}>0$.
Define
\begin{equation}\label{eq:semigroup}
C(w)=\NN g^{(1)}+\cdots+\NN g^{(m)}.
\end{equation}

For $a\in\NN^n$ and $z\in\CC^n$, set
\[
[z]_a=\prod_{i=1}^n\prod_{j=0}^{a_i-1}(z_i-j),
\quad \nsupp(z)=\{i\mid z_i\in\ZZ_{<0}\}.
\]
A vector $v\in\CC^n$ is a fake exponent with parameter $\beta$ if and
only if
\begin{equation}\label{eq:fake}
Av=\beta,\quad [v]_{g^{(k)}_+}=0\quad(1\le k\le m).
\end{equation}
Indeed, the reduced Gr\"obner basis of the distraction of $M$ consists
of the falling factorials corresponding to the minimal monomial
generators of $M$ \cite[Theorem~3.2.2]{SST}.
Thus the vanishing conditions for the fake indicial ideal are precisely
\eqref{eq:fake}.
Fix such a $v$ and set $I_0=\nsupp(v)$.
For each $k$, the second condition in \eqref{eq:fake} means that there
is an index $i$ such that
\begin{equation}\label{eq:root}
g_i^{(k)}>0,\quad
v_i\in\{0,\ldots,g_i^{(k)}-1\}.
\end{equation}

\begin{definition}\label{def:NS}
Define the negative support family by
\begin{equation}\label{eq:NS}
\NS(v)=\left\{I\ \middle|\
\begin{array}{l}
I=\nsupp(v+u)\text{ for some }u\in L,\text{ and}\\
\{u'\in L:\nsupp(v+u')=I\}\subseteq C(w)
\end{array}\right\},
\end{equation}
and set $K_v=\bigcap_{I\in\NS(v)}I$.
\end{definition}

This is the family in \cite[Equation~(9)]{Saito2020}.
Since $I_0\in\NS(v)$, the family whose intersection defines $K_v$ is
nonempty; see Proposition~\ref{prop:bounds}.

\subsection{Core and regular triangulations}

Define
\begin{equation}\label{eq:core}
C_w=\{i\mid g_i^{(k)}\le0\text{ for every }k\}.
\end{equation}
Thus $C_w$ is the set of indices of variables absent from every minimal
monomial generator of $M$.

A \emph{standard pair} of a monomial ideal $M$ is a pair $(a,\sigma)$,
where $a\in\NN^n$ and $\sigma\subseteq\{1,\ldots,n\}$, satisfying:
\begin{enumerate}[label=\textup{(\roman*)}]
\item $a_i=0$ for every $i\in\sigma$.
\item For every $b\in\NN^n$ supported on $\sigma$,
      $\partial^{a+b}\notin M$.
\item For each $l\notin\sigma$, there is a $b\in\NN^n$ supported on
      $\sigma\cup\{l\}$ such that $\partial^{a+b}\in M$.
\end{enumerate}

Write $S(M)$ for the set of standard pairs of $M$.
By \cite[Corollary~3.2.3]{SST}, the distraction of $M$ is
\[
\bigcap_{(a,\sigma)\in S(M)}
\langle\theta_j-a_j\mid j\notin\sigma\rangle.
\]
Adding the Euler equations $A\theta=\beta$, we see that $v$ is a fake
exponent if and only if $Av=\beta$ and there is a standard pair
$(a,\sigma)\in S(M)$ such that $v_j=a_j$ for every $j\notin\sigma$.

We have
\begin{equation}\label{eq:standardpairs}
C_w=\bigcap_{(a,\sigma)\in S(M)}\sigma.
\end{equation}
This also follows from the standard pair decomposition of a monomial
ideal \cite[Equation~(3.13)]{SST}.

Set
\begin{equation}\label{eq:Delta}
\Delta=\{\sigma\subseteq\{1,\ldots,n\}\mid
\supp(g^{(k)}_+)\not\subseteq\sigma\text{ for every }k\}.
\end{equation}
Then $\sqrt M$ is the Stanley--Reisner ideal of $\Delta$.
By the correspondence between toric initial ideals and regular
triangulations, $\Delta$ is the regular triangulation of
$\operatorname{conv}(a_1,\ldots,a_n)$ induced by $w$
\cite[Theorem~8.3]{Sturmfels}.
From \cite[Theorem 4.2 (i)]{NagamineHilbert}, the columns indexed by each face
of $\Delta$ are linearly independent.
In particular, every maximal face has $d$ indices.

Moreover,
\begin{equation}\label{eq:core-facets}
C_w\subseteq\sigma
\quad\text{for every maximal face }\sigma\text{ of }\Delta.
\end{equation}
Indeed, if $i\in C_w$, the variable $\partial_i$ occurs in no minimal
generator of $M$.
Hence $\sigma\in\Delta$ implies $\sigma\cup\{i\}\in\Delta$.
If $\sigma$ is maximal, it therefore contains $i$.

Define
\[
\operatorname{core}(\Delta)
:=\bigcap_{\substack{\sigma\in\Delta\\\sigma\text{ is maximal}}}\sigma.
\]
This is Stienstra's triangulation core \cite[Definition~3]{Stienstra},
and \eqref{eq:core-facets} gives
\[
C_w\subseteq\operatorname{core}(\Delta).
\]
Stienstra \cite[Lemma~6]{Stienstra} showed that a simplex not containing
$\operatorname{core}(\Delta)$ lies in the boundary of
$\operatorname{conv}(A)$.
The set $C_w$ used here is determined by all standard pairs, including
embedded ones, and can therefore be smaller than
$\operatorname{core}(\Delta)$.

We recall notations of the triangulation used below.
Let
\[
P=\operatorname{conv}(a_1,\ldots,a_n).
\]
Since $A$ is homogeneous and has rank $d$, the polytope $P$ has
dimension $d-1$.
For an index set $\sigma\subseteq\{1,\ldots,n\}$, write
\[
P_\sigma=\operatorname{conv}(a_i\mid i\in\sigma).
\]
If its vertices are affinely independent, $P_\sigma$ is a
\emph{simplex}.
The convex hull of any subset of its vertices is a face of the simplex.
A \emph{triangulation} of $P$ is a finite collection of simplices such
that:
\begin{enumerate}[label=\textup{(\roman*)}]
\item Their union is $P$.
\item The intersection of two distinct simplices is either empty or a
      common face of both.
\end{enumerate}
A simplex not properly contained in another simplex of the
triangulation is called a \emph{maximal simplex}.
Every maximal simplex in a triangulation of $P$ has dimension $d-1$
and has $d$ vertices.
A \emph{facet} of a polytope is a face of codimension one.
Thus the facets of $P$ have dimension $d-2$.

For a weight $w=(w_1,\ldots,w_n)$, lift each point $a_i$ to
\[
(a_i,w_i)\in\RR^{d+1}.
\]
The projections of the lower faces of the convex hull of these lifted
points give the \emph{regular subdivision} induced by $w$.
For a generic $w$, its maximal cells are simplices, and it is called
the \emph{regular triangulation} induced by $w$.

\subsection{Normal forms and basic properties of negative support}

We collect some basic facts about $C(w)$ and negative supports that will be used in the proof of the main theorem.
\begin{lemma}\label{lem:standard}
If $u\in L$ and $\partial^{u_-}\notin M$, then $u\in C(w)$.
\end{lemma}
\begin{proof}
The case $u=0$ is immediate.
For $u\ne0$, the binomial
$\partial^{u_+}-\partial^{u_-}$ belongs to $I_A$.
Since $\partial^{u_-}\notin M$, its initial term is
$\partial^{u_+}$.
Thus $\partial^{u_+}\in M$, and
\cite[Lemma~4.2]{Saito2020} gives $u\in C(w)$.
Equivalently, reduction by the reduced Gr\"obner basis expresses
$u=\sum_k c_kg^{(k)}$ with $c_k\in\NN$.
\end{proof}

\begin{lemma}\label{lem:subset}
Let $v$ be a fake exponent.
If $u\in L$ satisfies $\nsupp(v+u)\subseteq\nsupp(v)$, then
$u\in C(w)$.
\end{lemma}
\begin{proof}
By \cite[Proposition~4.4]{Saito2020}, $\nsupp(v+u)\in\NS(v)$.
Every shift realizing this negative support belongs to $C(w)$, thus in
particular $u\in C(w)$.
\end{proof}

\begin{proposition}\label{prop:bounds}
For every fake exponent $v$,
\begin{equation}\label{eq:bounds}
I_0\in\NS(v),\quad I_0\cap C_w\subseteq K_v\subseteq I_0.
\end{equation}
\end{proposition}
\begin{proof}
By \cite[Proposition~4.4]{Saito2020}, $I_0\in\NS(v)$, hence
$K_v\subseteq I_0$.
On the other hand, any shift $u$ realizing an $I\in\NS(v)$ belongs to
$C(w)$ and therefore satisfies $u_i\le0$ for $i\in C_w$.
Hence $I_0\cap C_w\subseteq I$ for every $I\in\NS(v)$, which gives
$I_0\cap C_w\subseteq K_v$.
\end{proof}

\begin{lemma}\label{lem:core-endpoint}
Let $v$ be a fake exponent.
If $u\in L$ satisfies $\nsupp(v+u)\subseteq C_w$, then $u\in C(w)$.
Consequently, every realizable negative support contained in $C_w$
belongs to $\NS(v)$.
\end{lemma}
\begin{proof}
Suppose that $\partial^{u_-}\in M$.
Then $\partial^{g^{(k)}_+}$ divides $\partial^{u_-}$ for some $k$.
Since $v$ is a fake exponent, there is an index $i$ satisfying
\eqref{eq:root}.
The inequality $g_i^{(k)}>0$ implies $i\notin C_w$.
The divisibility gives $u_i\le-g_i^{(k)}$, and hence
\[
v_i+u_i\le(g_i^{(k)}-1)-g_i^{(k)}=-1.
\]
The left-hand side is an integer, thus $i\in\nsupp(v+u)$, contrary to
the assumption.
Therefore $\partial^{u_-}\notin M$, and Lemma~\ref{lem:standard}
gives $u\in C(w)$.
The same argument applies to every shift realizing the given negative
support, proving the final assertion.
\end{proof}

\section{A geometric characterization}\label{sec:main}

For an index set $T$, write
$\cone(A_T)=\sum_{i\in T}\RR_{\ge0}a_i$, with
$\cone(A_\emptyset)=\{0\}$.
The cone $\cone(A)$ is generated by all columns of $A$, and
$\inte\cone(A)$ denotes its interior.
Since $A$ has rank $d$, the interior is taken in $\RR^d$.

A \emph{face} of $\cone(A)$ is a set of the form
\[
F=\{x\in\cone(A)\mid\lambda(x)=0\},
\]
where $\lambda$ is a linear form nonnegative on the cone.
If $\lambda=0$, then $F=\cone(A)$; if $\lambda\ne0$, then
$F\subsetneq\cone(A)$.

\Needspace{14\baselineskip}
\begin{theorem}\label{thm:main}
Fix $A$ and $w$ as in Section~\ref{sec:prelim}.
The following conditions are equivalent.
\begin{enumerate}[label=\textup{(\roman*)}]
\item For every $\beta\in\CC^d$ and every fake exponent $v$ with
      parameter $\beta$,
      \begin{equation}\label{eq:main}
      K_v=\nsupp(v)\cap C_w.
      \end{equation}
\item The cone generated by the core columns meets the interior of
      $\cone(A)$; that is,
      \[
      \cone(A_{C_w})\cap\inte\cone(A)\ne\emptyset.
      \]
\item There is an $h\in C(w)$ such that $h_i>0$ for every
      $i\notin C_w$.
\item The only face of $\cone(A)$ containing all core columns is
      $\cone(A)$ itself.
\end{enumerate}
The vector $h$ in \textup{(iii)} can be chosen using only $A$ and $w$,
and it satisfies $h_i\le0$ for every $i\in C_w$.
\end{theorem}

\subsection{Equivalence of the geometric conditions and sufficiency}

\begin{lemma}\label{lem:positive}
Conditions \textup{(ii)}, \textup{(iii)}, and \textup{(iv)} of
Theorem~\ref{thm:main} are equivalent.
\end{lemma}
\begin{proof}
We first prove the equivalence of \textup{(ii)} and \textup{(iv)}.
Set
\[
C=\cone(A),\quad C_0=\cone(A_{C_w}),\quad
b=\sum_{i\in C_w}a_i.
\]
We claim that
\[
C_0\cap\inte C\ne\emptyset
\quad\Longleftrightarrow\quad b\in\inte C.
\]
Since $b\in C_0$, the reverse implication is immediate.
For the forward implication, assume that $C_0\cap\inte C\ne\emptyset$
but $b\notin\inte C$.
Then $b$ is on the boundary of $C$, hence there is a nonzero linear form
$\lambda$ nonnegative on $C$ such that $\lambda(b)=0$.
Since
\[
0=\lambda(b)=\sum_{i\in C_w}\lambda(a_i),
\quad\lambda(a_i)\ge0,
\]
we have $\lambda(a_i)=0$ for every $i\in C_w$.
Thus
\[
C_0\subseteq F:=\{x\in C\mid\lambda(x)=0\}.
\]
The face $F$ is proper because $\lambda\ne0$.
Hence $F\cap\inte C=\emptyset$, contradicting
$C_0\cap\inte C\ne\emptyset$.
This proves the claim.

We next show that $b\in\inte C$ is equivalent to \textup{(iv)}.
If $b\notin\inte C$, the preceding argument gives a proper face
containing all core columns, thus \textup{(iv)} fails.
Conversely, if \textup{(iv)} fails, there is a proper face $F$ of $C$
containing all core columns.
Then
\[
b=\sum_{i\in C_w}a_i\in F.
\]
Since $F$ is contained in the boundary of $C$, we have
$b\notin\inte C$.
Therefore
\[
\textup{(ii)}\quad\Longleftrightarrow\quad
b\in\inte\cone(A)\quad\Longleftrightarrow\quad\textup{(iv)}.
\]

We now prove \textup{(ii)} $\Rightarrow$ \textup{(iii)}.
If $C_w=\{1,\ldots,n\}$, take $h=0$.
Otherwise, set $S=\sum_{i\notin C_w}a_i$.
Since $b$ is an interior point, for a sufficiently small rational
number $\varepsilon>0$ we have $b-\varepsilon S\in\cone(A)$.
The columns $a_i$ are integer vectors, and
$b-\varepsilon S\in\QQ^d$.
We may therefore choose $r_i\in\QQ_{\ge0}$ such that
\[
b-\varepsilon S=\sum_{i=1}^n r_i a_i.
\]
Define
\[
\widetilde h_i=
\begin{cases}
r_i+\varepsilon,&i\notin C_w,\\
r_i-1,&i\in C_w.
\end{cases}
\]
Then $A\widetilde h=0$, and $\widetilde h_i>0$ outside the core.
Clearing denominators gives an $h\in L$ with the same sign condition.
Since $\supp(h_-)\subseteq C_w$, we have
$\partial^{h_-}\notin M$.
Lemma~\ref{lem:standard} therefore gives $h\in C(w)$.
The generators of $C(w)$ are nonpositive on the core, hence $h_i\le0$
also holds for $i\in C_w$.

Finally, assume \textup{(iii)}.
The equality $Ah=0$ gives
\[
\sum_{i\notin C_w}h_i a_i+\sum_{i\in C_w}a_i
=\sum_{i\in C_w}(1-h_i)a_i.
\]
The left-hand side is a positive linear combination of all columns,
thus it belongs to $\inte\cone(A)$.
The right-hand side belongs to $\cone(A_{C_w})$.
Thus \textup{(ii)} holds.
\end{proof}

\begin{proof}[Proof of \textup{(iii)} $\Rightarrow$ \textup{(i)} in
Theorem~\ref{thm:main}]
Fix a fake exponent $v$ and set $I_0=\nsupp(v)$.
If $I_0\subseteq C_w$, Proposition~\ref{prop:bounds} gives
\[
K_v=I_0=I_0\cap C_w.
\]
Suppose now that $I_0\setminus C_w\ne\emptyset$.
Choose $h$ as in \textup{(iii)} and set
\[
N=\max_{i\in I_0\setminus C_w}(-v_i)\in\ZZ_{>0}.
\]
For $i\in I_0\setminus C_w$, the integer $h_i$ is positive, hence
$h_i\ge1$ and
\[
v_i+Nh_i\ge v_i+N\ge0.
\]
If $i\notin C_w$ and $v_i\in\ZZ_{\ge0}$, then $v_i+Nh_i\ge0$.
If $v_i\notin\ZZ$, then $v_i+Nh_i\notin\ZZ$ because $Nh_i\in\ZZ$.
Consequently,
\[
I:=\nsupp(v+Nh)\subseteq C_w.
\]
Lemma~\ref{lem:core-endpoint} gives $I\in\NS(v)$.
By Proposition~\ref{prop:bounds},
$I_0\cap C_w\subseteq K_v$.
On the other hand, since $I_0,I\in\NS(v)$,
\[
K_v\subseteq I_0\cap I\subseteq I_0\cap C_w.
\]
Therefore
\[
K_v=I_0\cap C_w=\nsupp(v)\cap C_w.
\]
\end{proof}

\subsection{Boundary counterexamples and necessity}

We prove the remaining implication \textup{(i)} $\Rightarrow$
\textup{(iv)} by contraposition.
Assuming that all core columns lie in a proper face of $\cone(A)$, we
construct a fake exponent that does not satisfy \eqref{eq:main}.

\begin{lemma}\label{lem:boundary-simplex}
Suppose that a proper face of $\cone(A)$ contains all core columns.
Then there exist a facet $F$ of $\cone(A)$, a face $\rho$ of $\Delta$,
a maximal face $\sigma$ of $\Delta$, and an index $j$ such that:
\begin{enumerate}[label=\textup{(\alph*)}]
\item $F$ contains all core columns and all columns indexed by
      $\rho$, and $|\rho|=d-1$.
\item $\sigma$ is the unique maximal face containing $\rho$, and
      \[
      \sigma=\rho\cup\{j\}.
      \]
\item $a_j\notin F$ and $j\notin C_w$.
\end{enumerate}
\end{lemma}
\begin{proof}
The hypothesis cannot hold when $d=1$.
Indeed, homogeneity implies that all columns of $A$ are the same
nonzero vector.
Since $w$ is generic, there is a unique index $j$ minimizing $w_j$, and
\[
M=\langle\partial_i\mid i\ne j\rangle,
\quad C_w=\{j\}.
\]
The only proper face of the one-dimensional cone $\cone(A)$ is
$\{0\}$, which does not contain $a_j$.

Assume henceforth that $d\ge2$.
Let $G$ be a proper face containing all core columns, and choose a
facet $F$ of $\cone(A)$ containing $G$.
Then $F$ also contains all core columns.
Set $P=\operatorname{conv}(A)$.
By homogeneity, $\dim P=d-1$, and $F\cap P$ is a facet of $P$.
Hence
\[
\dim(F\cap P)=d-2.
\]
Restrict $\Delta$ to $F\cap P$ and choose a maximal simplex, with
index set $\rho$, in this restriction.
Then
\[
\dim P_\rho=d-2,\quad |\rho|=d-1.
\]
As a boundary simplex of codimension one, $\rho$ belongs to exactly
one maximal face $\sigma$ of $\Delta$.
Since $\sigma$ has $d$ vertices, it has the form
\[
\sigma=\rho\cup\{j\}
\]
for some $j\notin\rho$.
All columns indexed by $\rho$ lie in $F$.
If $a_j\in F$, then $P_\sigma\subseteq F\cap P$, contradicting
\[
\dim P_\sigma=d-1>\dim(F\cap P)=d-2.
\]
Thus $a_j\notin F$.
Since all core columns lie in $F$, it follows that $j\notin C_w$.
\end{proof}

\begin{proposition}\label{3.4}
Let $F,\rho,\sigma,j$ be as in Lemma~\ref{lem:boundary-simplex}.
For $q\in\ZZ_{>0}$, define $v\in\CC^n$ by
\begin{equation}\label{eq:bad-v}
v_j=-q,\quad
v_k\in\CC\setminus\ZZ\quad(k\in\rho),\quad
v_i=0\quad(i\notin\sigma).
\end{equation}
Then $v$ is a fake exponent with parameter $\beta=Av$, and
\begin{equation}\label{eq:bad-result}
K_v=\nsupp(v)=\{j\},\quad
\nsupp(v)\cap C_w=\emptyset.
\end{equation}
\end{proposition}
\begin{proof}
Since $\sigma$ is a maximal face of $\Delta$, $(0,\sigma)$ is a
standard pair of $M$.
As $v_i=0$ for $i\notin\sigma$, the vector $v$ is a fake exponent
with parameter $\beta=Av$.
Equation~\eqref{eq:bad-v} and $\sigma=\rho\cup\{j\}$ give
\[
\nsupp(v)=\{j\}.
\]
By Proposition~\ref{prop:bounds}, $K_v\subseteq\{j\}$.
It remains to prove $j\in K_v$.

Choose a nonzero linear form $\eta$ defining $F$ such that
\[
\eta\ge0\quad\text{on }\cone(A),\quad
F=\{x\in\cone(A)\mid\eta(x)=0\}.
\]
Lemma~\ref{lem:boundary-simplex} gives
\[
\eta(a_k)=0\quad(k\in\rho),\quad \eta(a_j)>0.
\]
The $d$ columns indexed by the maximal face $\sigma$ are linearly
independent.
Thus, for each $i\notin\sigma$, there is a unique expression
\[
a_i=\sum_{k\in\sigma}\lambda_{ki}a_k,
\quad\lambda_{ki}\in\QQ.
\]
Applying $\eta$ gives
\[
\lambda_{ji}=\frac{\eta(a_i)}{\eta(a_j)}\ge0.
\]
Choose $N_i\in\ZZ_{>0}$ clearing the denominators and set
\[
t^{(i)}=N_i e_i-\sum_{k\in\sigma}N_i\lambda_{ki}e_k.
\]
Then $t^{(i)}\in L$ and
\[
\supp(t^{(i)}_-)\subseteq\sigma.
\]
Since $\sigma\in\Delta$, we have
$\partial^{t^{(i)}_-}\notin M$.
Lemma~\ref{lem:standard} gives $t^{(i)}\in C(w)$.
The vector $t^{(i)}$ is nonzero, thus
\[
w\cdot t^{(i)}>0.
\]

Now let $I=\nsupp(v+u)$ for some $u\in L$, and suppose that
$j\notin I$.
For $k\in\rho$, the coordinate $v_k$ is not an integer, hence
$I\cap\rho=\emptyset$.
Consequently,
\[
I\subseteq\{1,\ldots,n\}\setminus\sigma.
\]
Since $Au=0$,
\[
\eta\bigl(A(v+u)\bigr)=\eta(Av)=-q\,\eta(a_j)<0.
\]
Because $j\notin I$, we have $v_j+u_j\ge0$.
For $i\notin\sigma\cup I$, we also have $u_i\ge0$.
In the sum
\[
\eta\bigl(A(v+u)\bigr)
=\sum_{\ell=1}^n(v_\ell+u_\ell)\eta(a_\ell),
\]
the terms indexed by $\rho$ vanish, and those indexed by $j$ or by
$\ell\notin\sigma\cup I$ are nonnegative.
If $\eta(a_i)=0$ for every $i\in I$, the sum would be nonnegative,
a contradiction.
Hence $\eta(a_i)>0$ for some $i\in I$, and therefore
\[
S:=\sum_{i\in I}N_i\lambda_{ji}>0.
\]
By the choice of $N_i$, the number $S$ is an integer, thus $S\ge1$.

Set
\[
u'=-q\sum_{i\in I}t^{(i)}.
\]
Outside $\sigma$,
\[
(v+u')_i=
\begin{cases}
-qN_i,&i\in I,\\
0,&i\notin I,
\end{cases}
\]
while
\[
(v+u')_j=q(S-1)\ge0.
\]
For $k\in\rho$, the coordinate $v_k+u'_k$ is not an integer.
It follows that
\[
\nsupp(v+u')=I.
\]
On the other hand,
\[
w\cdot u'=-q\sum_{i\in I}w\cdot t^{(i)}<0,
\]
hence $u'\notin C(w)$.
Thus a negative support $I$ not containing $j$ cannot belong to
$\NS(v)$.
Every member of $\NS(v)$ therefore contains $j$, proving $j\in K_v$.
Hence
\[
K_v=\{j\}=\nsupp(v).
\]
Finally, $j\notin C_w$ gives $\nsupp(v)\cap C_w=\emptyset$.
\end{proof}

By Proposition~\ref{3.4}, if all core columns lie in a proper face of
$\cone(A)$, there is a fake exponent for which \eqref{eq:main} fails.
This proves \textup{(i)} $\Rightarrow$ \textup{(iv)} and completes
the proof of Theorem~\ref{thm:main}.
In particular, under the conditions of the theorem, every fake
exponent $v$ satisfies
\[
K_v=\nsupp(v)
\quad\Longleftrightarrow\quad\nsupp(v)\subseteq C_w.
\]

\subsection{A criterion from the initial ideal}

\begin{corollary}\label{cor:combinatorial}
The conditions of Theorem~\ref{thm:main} are equivalent to the following
condition:
for every face $\rho\in\Delta$ with $|\rho|=d-1$ that is contained in
exactly one maximal face $\sigma$, writing
$\sigma\setminus\rho=\{j\}$, we have
\[
j\in C_w.
\]
\end{corollary}
\begin{proof}
Assume \textup{(iv)} of Theorem~\ref{thm:main}.
A face $\rho$ as in the statement lies on the boundary of the
triangulation.
Suppose that $j\notin C_w$.
By \eqref{eq:core-facets}, $C_w\subseteq\sigma$.
Since $\sigma=\rho\cup\{j\}$, it follows that
\[
C_w\subseteq\rho.
\]
The columns indexed by the boundary face $\rho$ lie in a proper face
of $\cone(A)$.
All core columns therefore lie in that face, contradicting
\textup{(iv)}.

Conversely, if \textup{(iv)} fails, there is a proper face of
$\cone(A)$ containing all core columns.
Lemma~\ref{lem:boundary-simplex} then gives $\rho$ and $\sigma$ with
$j\notin C_w$, thus the stated condition fails.
\end{proof}

\begin{remark}
If embedded standard pairs are present, $C_w$ need not equal the
intersection of the maximal faces of $\Delta$.
Alternatively, one can test whether the rational system of inequalities
\begin{equation}\label{eq:LP}
c_k\ge0,\quad\sum_k c_kg_i^{(k)}>0\quad(i\notin C_w)
\end{equation}
has a solution.
If a real solution exists, a rational solution exists as well.
Clearing denominators then gives the vector $h$ in
\textup{(iii)} of Theorem~\ref{thm:main}.
\end{remark}

\section{Examples}\label{sec:examples}

\subsection{A direct verification of the main equality}\label{ex:main-square}

Let
\[
A=\begin{pmatrix}
1&1&1&1\\
0&1&1&0\\
0&0&1&1
\end{pmatrix},\quad w=(0,2,0,1).
\]
Set $g=(-1,1,-1,1)$.
Then
\[
L=\ZZ g,\quad
I_A=\langle\partial_2\partial_4-\partial_1\partial_3\rangle,
\quad\inw I_A=\langle\partial_2\partial_4\rangle,
\]
thus $C(w)=\NN g$ and $C_w=\{1,3\}$.
Since $g_2=g_4=1>0$, the vector $h=g$ satisfies
\textup{(iii)} of Theorem~\ref{thm:main}.
Thus, for every fake exponent $v$,
\[
K_v=\nsupp(v)\cap\{1,3\}.
\]
Geometrically, the midpoint of the segment joining the two core
columns is
\[
\frac{a_1+a_3}{2}=\left(1,\frac12,\frac12\right),
\]
which lies in the relative interior of the square
$\operatorname{conv}(a_1,a_2,a_3,a_4)$.
This also verifies \textup{(ii)}.

For example, $v=(-1,0,0,-1)$ is a fake exponent with parameter
$\beta=Av=(-2,0,-1)$.
Every lattice shift has the form $kg$ for $k\in\ZZ$, and
\[
v+kg=(-1-k,k,-k,-1+k),\quad
\nsupp(v+kg)=
\begin{cases}
\{2,4\},&k\le-1,\\
\{1,4\},&k=0,\\
\{1,3\},&k\ge1.
\end{cases}
\]
Since $kg\in C(w)$ if and only if $k\ge0$, the definition gives
\[
\NS(v)=\bigl\{\{1,4\},\{1,3\}\bigr\},\quad
K_v=\{1\}=\nsupp(v)\cap C_w.
\]
Here $\nsupp(v)$ is not contained in the core, but its intersection
with the core is precisely $K_v$.

\subsection{A hexagonal configuration where the condition fails}
\label{ex:hexagon}

Consider the same configuration as in
\cite[Example~3.3]{NagamineCore}:
\[
A=\begin{pmatrix}
1&1&1&1&1&1\\
0&1&1&0&-1&-1\\
-1&-1&0&1&1&0
\end{pmatrix},\quad
w=(2,3,8,1,13,5).
\]
In this case,
\[
\begin{aligned}
\inw I_A=\langle\,
&\partial_2\partial_5,\,
\partial_3\partial_6,\,
\partial_3\partial_5^2,\,
\partial_3^2\partial_5,\,
\partial_1\partial_5^2,\\
&\partial_1\partial_3\partial_5,\,
\partial_1\partial_3^2,\,
\partial_1^2\partial_5,\,
\partial_2^2\partial_6^2,\,
\partial_1^2\partial_3
\,\rangle.
\end{aligned}
\]
Consequently,
\[
C_w=\{4\}.
\]
The maximal faces of $\Delta$ are
\[
\{1,2,4\},\quad\{1,4,6\},\quad
\{2,3,4\},\quad\{4,5,6\}.
\]
Set
\[
\rho=\{3,4\},\quad\sigma=\{2,3,4\}.
\]
The face $\rho$ is contained in the unique maximal face $\sigma$, and
\[
\sigma\setminus\rho=\{2\}.
\]
Since $2\notin C_w$, Corollary~\ref{cor:combinatorial} shows that the
conditions of Theorem~\ref{thm:main} fail.

This failure can be seen explicitly in a family of fake exponents.
Let $F$ be the facet of $\cone(A)$ generated by $a_3$ and $a_4$.
For $q\in\ZZ_{>0}$, set
\[
v^{(q)}=\left(0,-q,\frac12,\frac12,0,0\right).
\]
Applying Proposition~\ref{3.4} to $F,\rho,\sigma$, and $j=2$, we see
that $v^{(q)}$ is a fake exponent with parameter
\[
\beta^{(q)}=Av^{(q)}
=\left(1-q,\frac12-q,q+\frac12\right)^T,
\]
and
\[
K_{v^{(q)}}=\nsupp(v^{(q)})=\{2\}.
\]
However,
\[
\nsupp(v^{(q)})\cap C_w=\{2\}\cap\{4\}=\emptyset.
\]
Therefore
\[
K_{v^{(q)}}\ne\nsupp(v^{(q)})\cap C_w.
\]
Thus the hexagonal configuration gives an example where the conditions
of Theorem~\ref{thm:main} do not hold.

\section*{Use of AI}

In preparing this paper, the author used ChatGPT-5.6 Sol in the
following ways:
\begin{itemize}
\item Some ideas incorporated into the proofs in Section 3 were
suggested by ChatGPT-5.6 Sol.
\item ChatGPT-5.6 Sol was used to improve the exposition and to help
identify grammatical errors and possible mathematical errors.
\end{itemize}
All AI-generated suggestions and text incorporated into the paper were
carefully checked, independently verified, and revised by the author.
The author takes full responsibility for the correctness of all mathematical
statements and proofs and for the final text of this paper.

\subsection*{\normalsize Mao NAGAMINE}
\noindent
Department of Mathematics \\
Graduate School of Science\\
Hokkaido University\\
Sapporo 060-0810, JAPAN\\
e-mail: nagamine.mao.d9@elms.hokudai.ac.jp

\end{document}